\documentclass[11pt,draft]{amsart}

\usepackage{a4,amssymb,amsmath,cite}

\DeclareMathOperator{\partition}{part}

\DeclareMathOperator{\var}{var}

\newtheorem{theorem}{Theorem}[section]

\newtheorem{proposition}[theorem]{Proposition}

\newtheorem{lemma}[theorem]{Lemma}

\newtheorem{corollary}[theorem]{Corollary}

\newtheorem{problem}[theorem]{Problem}

\theoremstyle{definition}

\newtheorem{example}[theorem]{Example}

\makeatletter

\renewcommand*\subjclass[2][2000]{\def\@subjclass{#2}\@ifundefined
{subjclassname@#1}{\ClassWarning{\@classname}{Unknown edition (#1) of
Mathematics Subject Classification; using '2000'.}}{\@xp\let\@xp
\subjclassname\csname subjclassname@#1\endcsname}}

\@addtoreset{equation}{section}

\makeatother

\begin{document}

\title[Special elements in the lattice of overcommutative varieties]{Special
elements in the lattice\\
of overcommutative semigroup varieties\\
revisited}

\author{V. Yu. Shaprynski\v{\i}}

\address{Department of Mathematics and Mechanics, Ural State University,
Lenina 51, 620083 Ekaterinburg, Russia}

\email{vshapr@yandex.ru}

\author{B. M. Vernikov}

\address{Department of Mathematics and Mechanics, Ural State University,
Lenina 51, 620083 Ekaterinburg, Russia}

\email{boris.vernikov@usu.ru}

\thanks{The work was partially supported by the Russian Foundation for Basic
Research (grant No.~09-01-12142).}

\date{}

\begin{abstract}
We completely determine all distributive, codistributive, standard,
costandard, and neutral elements in the lattice of overcommutative semigroup
varieties, thus correcting a gap contained in \cite{Vernikov-01}.
\end{abstract}



\maketitle

\section{Introduction}
\label{intr}

The class of all semigroup varieties forms a lattice under the following
naturally defined operations: for varieties $\mathcal X$ and $\mathcal Y$,
their \emph{join} $\mathcal{X\vee Y}$ is the variety generated by the
set-theoretical union of $\mathcal X$ and $\mathcal Y$ (as classes of
semigroups), while their \emph{meet} $\mathcal{X\wedge Y}$ coincides with the
class-theoretical intersection of $\mathcal X$ and $\mathcal Y$. This lattice
has been intensively studied for about four decades. A systematic overview of
the material accumulated here is given in the recent survey \cite
{Shevrin-Vernikov-Volkov-09}.

It is a common knowledge that the lattice \textbf{SEM} of all semigroup
varieties is divided into two large sublattices with essentially different
properties: the coideal \textbf{OC} of all \emph{overcommutative} varieties
(that is, varieties containing the variety of all commutative semigroups) and
the ideal of all \emph{periodic} varieties (that is, varieties consisting of
periodic semigroups).

The global structure of the lattice \textbf{OC} has been revealed by Volkov
in~\cite{Volkov-94}. It is proved there that this lattice decomposes into a
subdirect product of its certain intervals and each of these intervals is
anti-isomorphic to the congruence lattice of a certain unary algebra of a
special type (namely, of a so-called $G$-set; a basic information about
$G$-sets see in~\cite{McKenzie-McNulty-Taylor-87}, for instance). The exact
formulation of this result may be found also in~\cite[Theorem~5.1]
{Shevrin-Vernikov-Volkov-09}. We do not reproduce this formulation here
because we do not use it below.

There are several articles where special elements of different types in the
lattice \textbf{SEM} have been examined (see~\cite{Jezek-McKenzie-93,
Vernikov-07-cmod,Vernikov-07-lmod,Vernikov-08-lmod,Vernikov-08-umod1,
Vernikov-08-umod2,Vernikov-codistr,Vernikov-Shaprynskii-distr,
Vernikov-Volkov-06,Volkov-05}). We refer an interested reader to~\cite
[Section~14]{Shevrin-Vernikov-Volkov-09} for an overview of the most part of
results obtained in these articles.

Recall that an element $x$ of a lattice $\langle L;\ \vee,\wedge\rangle$ is
called \emph{distributive} if
$$\forall y,z\in L\colon\quad x\vee(y\wedge z)=(x\vee y)\wedge(x\vee z);$$
\emph{standard} if
$$\forall y,z\in L\colon\quad(x\vee y)\wedge z=(x\wedge z)\vee(y\wedge z);$$
\emph{neutral} if, for all $y,z\in L$, the sublattice of $L$ generated by
$x$, $y$, and $z$ is distributive. \emph{Codistributive} [\emph
{costandard}\,] elements are defined dually to distributive [respectively
standard] ones. An extensive information about elements of all these five
types in abstract lattices may be found in~\cite[Section~III.2]{Gratzer-98},
for instance. Note that any [co]standard element is [co]distributive, and an
element is neutral if and only if it is standard and costandard
simultaneously (see~\cite [Theorem~III.2.5]{Gratzer-98}, for instance). On
the other hand, a [co]distributive element may be not [co]standard, while a
[co]standard element may be not neutral.

A complete description of neutral elements in the lattice \textbf{SEM} has
been given in~\cite[Proposition~4.1]{Volkov-05} (see also~\cite[Theorem~14.2]
{Shevrin-Vernikov-Volkov-09}). In~\cite{Vernikov-Shaprynskii-distr}, all
distributive elements in \textbf{SEM} are completely determined. In~\cite
{Vernikov-codistr}, quite a strong necessary condition for semigroup
varieties to be a codistributive element in \textbf{SEM} is obtained. In
particular, all varieties with each of these three properties (except the
trivial extreme case of the variety $\mathcal{SEM}$ of all semigroups) turn
out to be periodic varieties.

So, an examination of special elements of all the mentioned types in the
lattice \textbf{SEM} gives no any information concerning the lattice \textbf
{OC}. Aiming to obtain some new knowledge about this lattice, it is natural
to investigate its special elements.

Such investigations have been started by the second author in~\cite
{Vernikov-01}. Five types of special elements (namely, distributive,
codistributive, standard, costandard, and neutral elements) in the lattice
\textbf{OC} have been considered there. Unfortunately, it turns out that
considerations in~\cite{Vernikov-01} contain a gap, and the main result of
this article is incorrect. Namely, it was proved in~\cite{Vernikov-01} that,
for an overcommutative semigroup variety, the properties of being a
distributive element of \textbf {OC}, of being a codistributive element of
\textbf {OC}, of being a standard element of \textbf {OC}, of being a
costandard element of \textbf{OC}, and of being a neutral element of \textbf
{OC} are equivalent. This result of~\cite{Vernikov-01} is true. But, besides
that, the main result of~\cite{Vernikov-01} contains a list of all
overcommutative varieties that possess the five mentioned properties.
Unfortunately, this list turns out to be non-complete. All varieties from the
list really have all the mentioned properties, but there are many other such
varieties. The objective of this article is to give a correct description of
distributive, codistributive, standard, costandard, and neutral elements in
the lattice \textbf{OC}.

The article is structured as follows. In Section~\ref{prel}, we introduce a
necessary notation and formulate the main result of the article (Theorem~\ref
{main}). In Section~\ref{order}, we prove several auxiliary facts.
Sections~\ref{nec} and~\ref{suff} are devoted to the proof of Theorem~\ref
{main}. In Section~\ref{optim}, we show that this theorem can not be
improved, in a sense. Finally, in Section~\ref{quest}, we formulate some open
problems.

\section{Preliminaries and summary}
\label{prel}

We denote by $F$ the free semigroup over a countably infinite alphabet
$\{x_1,x_2$, $\dots,x_n,\dots\}$. As usual, elements of $F$ are called \emph
{words}. By $F^1$ we denote the semigroup $F$ with the empty word ajoined.
The symbol $\equiv$ stands for the equality relation on $F$ and $F^1$. If $u$
is a word, then $\ell(u)$ denotes the length of $u$, $\ell_i(u)$ is the
number of occurrences of the letter $x_i$ in $u$, $c(u)$ stands for the set
of all letters occurring in $u$, and $n(u)=|c (u)|$ is the number of letters
occurring in $u$. An identity $u\approx v$ is called \emph{balanced} if
$\ell_i(u)=\ell_i(v)$ for all $i$. It is a common knowledge that if an
overcommutative variety satisfies some identity then this identity is
balanced.

Let $m$ and $n$ be integers with $2\le m\le n$. A \emph{partition of the
number $n$ into $m$ parts} is a sequence of positive integers $(\ell_1,
\ell_2,\dots,\ell_m)$ such that
$$\ell_1\ge\ell_2\ge\cdots\ge\ell_m\quad\text{and}\quad\sum_{i=1}^m\ell_i=n
\ldotp$$
The numbers $\ell_1,\ell_2,\dots,\ell_m$ are called \emph{components} of the
partition $\lambda$. We denote by $\Lambda_{n,m}$ the set of all partitions
of the number $n$ into $m$ parts and by $\Lambda$ the union of the sets
$\Lambda_{n,m}$ for all natural numbers $m$ and $n$ with $2\le m\le n$. If
$\lambda\in\Lambda_{n,m}$ then we denote the numbers $n$ and $m$ by $n
(\lambda)$ and $m(\lambda)$ respectively.

If $u$ is a word then we denote by $\partition(u)$ the partition of the
number $\ell(u)$ into $n(u)$ parts consisting of integers $\ell_i(u)$ for all
$i$ such that $x_i\in c(u)$ (the numbers $\ell_i(u)$ are placed in
$\partition(u)$ in non-increasing order). If $u\approx v$ is a balanced
identity then, obviously, $\ell(u)=\ell(v)$, $n(u)=n(v)$, and $\partition(u)=
\partition(v)$. We call the partition $\partition(u)$ a \emph{partition} of
the identity $u\approx v$. We denote the numbers $\ell(u)=\ell(v)$ and $n(u)=
n(v)$ by $\ell(u\approx v)$ and $n(u\approx v)$ respectively, and the
partition $\partition(u)=\partition(v)$ by $\partition(u\approx v)$.

Let $\lambda=(\ell_1,\ell_2,\dots,\ell_m)\in\Lambda_{n,m}$. We denote by
$W_{n,m,\lambda}$, or simply $W_\lambda$, the set of all words $u$ such that
$\ell(u)=n$, $c(u)=\{x_1,x_2,\dots,x_m\}$, $\ell_i(u)\ge\ell_{i+1}(u)$ for
all $i=1,2,\dots,m-1$, and $\partition(u)=\lambda$. It is evident that every
balanced identity $u\approx v$ with $\ell(u\approx v)=n$, $n(u\approx v)=m$,
and $\partition(u\approx v)=\lambda$ is equivalent to some identity $s\approx
t$ where $s,t\in W_{n,m,\lambda}$.

We call sets of the kind $W_{n,m,\lambda}$ \emph{transversals}. We say that
an overcommutative variety $\mathcal V$ \emph{reduces} [\emph{collapses}] a
transversal $W_{n,m,\lambda}$ if $\mathcal V$ satisfies some non-trivial
identity [all identities] of the kind $u\approx v$ with $u,v\in
W_{n,m,\lambda}$. An overcommutative variety $\mathcal V$ is said to be \emph
{greedy} if it collapses any transversal it reduces. The following assertion
has been proved in~\cite{Vernikov-01}.

\begin{proposition}
\label{greedy}
An overcommutative semigroup variety is a distributive \textup
[co\-distributive, standard, costandard, neutral\textup] element of the
lattice $\mathbf{OC}$ if and only if it is greedy.\qed
\end{proposition}

This assertion was not formulated in~\cite{Vernikov-01} explicitly but it
directly follows from the proof of Theorem~2 in~\cite{Vernikov-01} (and
the corresponding part of the proof in~\cite{Vernikov-01} is correct).

It is an appropriate place here to indicate the error made in~\cite
{Vernikov-01}. Let $m$ and $n$ be positive integers with $2\le m\le n$ and
$\lambda\in\Lambda_{n,m}$. A semigroup variety given by an identity system
$\Sigma$ is denoted by $\var\Sigma$. We put
\begin{align*}
&\mathcal X_n=\var\,\{u\approx v\mid\ \text{the identity}\ u\approx v\ \text
{is balanced and}\ \ell(u\approx v)\ge n\},\\
&\mathcal X_{n,m}=\mathcal X_{n+1}\wedge\var\,\{u\approx v\mid\ \text{the
identity}\ u\approx v\ \text{is balanced},\ \ell(u\approx v)=n,\\
&\phantom{\mathcal X_{n,m}=\mathcal X_{n+1}\wedge\var\,\{u\approx v\mid\ \ }
\text{and}\ n(u\approx v)\le m\},\\
&\mathcal X_{n,1}=\mathcal X_{n+1},\\
&\mathcal X_{n,m,\lambda}=\mathcal X_{n,m-1}\wedge\var\{u\approx v\mid u,v\in
W_{n,m,\lambda}\}\ldotp
\end{align*}
It is claimed in~\cite{Vernikov-01} without any proof that an overcommutative
variety is greedy if and only if it coincides with one of the varieties
$\mathcal{SEM}$, $\mathcal X_n$, $\mathcal X_{n,m}$ or $\mathcal
X_{n,m,\lambda}$. Combining this claim with Proposition~\ref{greedy}, we
obtain the main result of~\cite{Vernikov-01}: the varieties $\mathcal{SEM}$,
$\mathcal X_n$, $\mathcal X_{n,m}$, $\mathcal X_{n,m,\lambda}$, and only they
are [co]distributive, [co]standard, and neutral elements in \textbf{OC}. In
actual fact, it is true that all these varieties are elements of the
mentioned types in \textbf{OC}. But the list of [co]distributive,
[co]standard, and neutral elements in \textbf{OC} is not exhausted by the
varieties $\mathcal{SEM}$, $\mathcal X_n$, $\mathcal X_{n,m}$, and $\mathcal
X_{n,m,\lambda}$. There are many other varieties with such a property.
Exactly this fact has been so unfortunately overseen in~\cite{Vernikov-01}.

For a partition $\lambda=(\ell_1,\ell_2,\dots,\ell_m)\in\Lambda_{n,m}$, we
define numbers $q(\lambda)$, $r(\lambda)$, and $s(\lambda)$ by the following
way:
\begin{align*}
&q(\lambda)\ \text{is the number of}\ \ell_i\text{'s with}\ \ell_i=1\ \text
{(if}\ \ell_m>1\ \text{then}\ q(\lambda)=0\text{);}\\
&r(\lambda)\ \text{is the sum of all}\ \ell_i\text{'s with}\ \ell_i>1\ \text
{(if}\ \ell_1=1\ \text{then}\ r(\lambda)=0\text{);}\\
&s(\lambda)=\max\,\{r(\lambda)-q(\lambda)-\delta,0\}
\end{align*}
where
$$\delta=
\begin{cases}
0&\text{whenever}\ n=3,m=2,\ \text{and}\ \lambda=(2,1),\\
1&\text{otherwise}\ldotp
\end{cases}$$
If $k$ is a non-negative integer then $\lambda^k$ stands for the following
partition of $n+k$ into $m+k$ parts:
$$\lambda^k=(\ell_1,\ell_2,\dots,\ell_m,\underbrace{1,\dots,
1}_{k\ \text{times}})$$
(in particular, $\lambda^0=\lambda$).

For a partition $\lambda\in\Lambda_{n,m}$, we put
$$\mathcal W_{n,m,\lambda}=\var\,\{\,u\approx v\mid u,v\in W_{n,m,\lambda}\}
\quad\text{and}\quad\mathcal S_\lambda=\bigwedge_{i=0}^{s(\lambda)}\mathcal
W_{n+i,m+i,\lambda^i}\ldotp$$
Sometimes we will write $\mathcal W_\lambda$ rather than $\mathcal
W_{n,m,\lambda}$.

The main result of the article is the following

\begin{theorem}
\label{main}
For an overcommutative semigroup variety $\mathcal V$, the following are
equivalent:
\begin{itemize}
\item[(i)]$\mathcal V$ is a distributive element of the lattice $\mathbf
{OC}$;
\item[(ii)]$\mathcal V$ is a codistributive element of the lattice $\mathbf
{OC}$;
\item[(iii)]$\mathcal V$ is a standard element of the lattice $\mathbf{OC}$;
\item[(iv)]$\mathcal V$ is a costandard element of the lattice $\mathbf{OC}$;
\item[(v)]$\mathcal V$ is a neutral element of the lattice $\mathbf{OC}$;
\item[(vi)]either $\mathcal{V=SEM}$ or $\mathcal V=\bigwedge\limits_{i=1}^k
\mathcal S_{\lambda_i}$ for some partitions $\lambda_1,\lambda_2,\dots,
\lambda_k\in\Lambda$.
\end{itemize}
\end{theorem}

The following claim was formulated in \cite{Vernikov-01} as a corollary of
the main result of that article. Theorem \ref{main} shows that the claim is
correct.

\begin{corollary}
\label{countable}
The set of all \textup[co\textup]distributive elements of the lattice
$\mathbf{OC}$ is countably infinite.\qed
\end{corollary}

This corollary is of some interest because the set of all overcommutative
semigroup varieties is well known to be uncountably infinite. On the other
hand, it is interesting to note that the set of all neutral elements in the
lattice \textbf{OC} is infinite, while the set of all neutral elements in the
lattice \textbf{SEM} consists of 5 varieties only~\cite[Proposition~4.1]
{Volkov-05}.

In view of Proposition~\ref{greedy}, Theorem~\ref{main} is equivalent to the
following

\begin{proposition}
\label{reduce}
An overcommutative semigroup variety $\mathcal V$ satisfies the
condition~\textup{(vi)} of Theorem~\textup{\ref{main}} if and only if it is
greedy.
\end{proposition}

It is this claim that will be verified in Sections~\ref{nec} and~\ref{suff}
(in fact, we prove the `only if' and `if' parts of Proposition~\ref{reduce}
in Sections~\ref{nec} and~\ref{suff} respectively). To prepare this proof, we
introduce some order relation on the set $\Lambda$ and consider some
properties of this relation in Section~\ref{order}.

\section{An order relation on the set $\Lambda$}
\label{order}

Let $\lambda=(\ell_1,\ell_2,\dots,\ell_m)\in\Lambda_{n,m}$ where $m\ge3$ and
$1\le i<j\le m$. We denote by $U_{i,j}(\lambda)$ the partition of the number
$n$ into $m-1$ parts with the components $\ell_1$, $\ell_2$, \dots,
$\ell_{i-1}$, $\ell_{i+1}$, \dots, $\ell_{j-1}$, $\ell_{j+1}$, \dots,
$\ell_m$, and $\ell_i+\ell_j$ (these components are written in $U_{i,j}
(\lambda)$ in non-increasing order). We will say that the partition $U_{i,j}
(\lambda)$ is obtained from $\lambda$ by the \emph{union of components}
$\ell_i$ and $\ell_j$. The partitioin obtained from $\lambda$ by a finite
(may be empty) set $S$ of unions of components is denoted by $U_S(\lambda)$;
in particular, $U_\varnothing(\lambda)=\lambda$.

We introduce a binary relation $\preceq$ on the set $\Lambda$ by the
following rule:
$$\lambda\preceq\mu\ \text{if and only if}\ \mu=U_S\bigl(\lambda^k\bigr)\
\text{for some}\ S\ \text{and}\ k\ldotp$$
The principal property of the relation $\preceq$ is given by the following

\begin{lemma}
\label{le is order}
The relation $\preceq$ is a partial order on the set $\Lambda$.
\end{lemma}

\begin{proof}
Reflexivity of $\preceq$ is evident because $\lambda=U_\varnothing\bigl
(\lambda^0\bigr)$. The claim that $\preceq$ is transitive also is evident
because if $\mu=U_S\bigl(\lambda^k\bigr)$ and $\nu=U_T\bigl(\mu^\ell\bigr)$
then $\nu=U_{S\cup T}\bigl(\lambda^{k+\ell}\bigr)$. To prove that $\preceq$
is antisymmetric, we suppose that $\lambda\preceq\mu$ and $\mu\preceq\lambda$
for some $\lambda,\mu\in\Lambda$. Then $\mu=U_S\bigl(\lambda^k\bigr)$ and
$\lambda=U_T\bigl(\mu^\ell\bigr)$ for some $S$, $T$, $k$, and $\ell$. Let
$n(\lambda)=n$ and $n(\mu)=q$. Then $q=n+k$ and $n=q+\ell$. Therefore, $q=q+
k+\ell$, whence $k+\ell=0$. This means that $k=\ell=0$. Thus, $\mu=U_S
(\lambda)$ and $\lambda=U_T(\mu)$. If $S\ne\varnothing$ then $r<m$. But $m\le
r$ because $\lambda$ is obtained from $\mu$ by unions of components.
Therefore, $S=\varnothing$ and $\mu=U_\varnothing(\lambda)=\lambda$.
\end{proof}

Now we are going to show that the partial order $\preceq$ has some nice
properties that will be played the crucial role in Section~\ref{suff}. The
first such property is given by the following

\begin{lemma}
\label{min condition}
The partially ordered set $\langle\Lambda;\preceq\rangle$ satisfies the
descending chain condition.
\end{lemma}

\begin{proof}
Let $\lambda,\mu\in\Lambda$ and $\lambda\preceq\mu$. Put $n(\lambda)=n$ and
$n(\mu)=q$. Then $q\le n$. Evidently, the set
$$\bigcup_{\substack{q\le n\\
2\le r\le q}}\Lambda_{q,r}$$
is finite. Thus, there exists finitely many partitions $\mu$ with $\mu\preceq
\lambda$ only. This immediately implies the desirable conclusion.
\end{proof}

We define one more binary relation $\unlhd$ on the set $\Lambda$ by the
following rule. Let $\lambda,\nu\in\Lambda$, $\lambda=(\ell_1,\ell_2,\dots,
\ell_m)$, and $\nu=(n_1,n_2,\dots,n_k)$. Then $\lambda\unlhd\nu$ if and only
if $m\le k$ and $\ell_i\le n_i$ for all $i=1,2,\dots,m$. It is evident that
$\unlhd$ is a partial order on $\Lambda$. The following claim shows a
relationship between orders $\preceq$ and $\unlhd$.

\begin{lemma}
\label{2 orders}
Let $\lambda,\nu\in\Lambda$. If $\lambda\unlhd\nu$ then $\lambda\preceq\nu$.
\end{lemma}

\begin{proof}
Let $\lambda=(\ell_1,\ell_2,\dots,\ell_m)$ and $\nu=(n_1,n_2,\dots,n_k)$.
Then $m\le k$ and $\ell_i\le n_i$ for all $i=1,2,\dots,m$. Put $s=n(\nu)-n
(\lambda)$. It is evident that $s\ge0$. If $s=0$ then $\lambda=\nu$ and we
are done. Let now $s>0$. By the trivial induction, it suffices to consider
the case $s=1$. Then either $k=m+1$, $\ell_i=n_i$ for all $i=1,2,\dots,m$,
and $n_k=1$ or $k=m$, $n_i=\ell_i+1$ for some $i\in\{1,2,\dots,m\}$ and $n_j=
\ell_j$ for all $j\ne i$. It is evident that $\nu=U_\varnothing\bigl
(\lambda^1\bigr)$ in the former case, while $\nu=U_{i,m+1}\bigl(\lambda^1
\bigr)$ in the latter one. Thus, $\lambda\preceq\nu$ in any case.
\end{proof}

The second important property of the relation $\preceq$ is given by the
following

\begin{lemma}
\label{antichains}
The partially ordered set $\langle\Lambda;\preceq\rangle$ does not contain
infinite anti-chains.
\end{lemma}

\begin{proof}
Arguing by contradiction, suppose that $\Lambda$ contains an infinite
anti-chain $A_0$. Put $m_1=\min\,\{m(\lambda)\mid\lambda\in A_0\}$. Let us
fix a partition $\lambda_1=(\ell_1^1,\ell_2^1,\dots,\ell_{m_1}^1)\in A_0$. If
$\nu=(n_1,n_2,\dots,n_k)$ is an arbitrary partition from $A_0$ then
$\lambda_1\npreceq\nu$, whence $\lambda_1\ntrianglelefteq\nu$ by Lemma~\ref
{2 orders}. Since $m_1\le k$, this means that $\ell_i^1>n_i$ for some $i\in
\{1,2,\dots,m_1\}$. The set $A_0$ is infinite, while the index $i$ runs over
the finite set $\{1,2,\dots,m_1\}$. Hence there is an index $i_1\le m_1$ such
that $n_{i_1}<\ell_{i_1}^1$ for an infinite set of partitions $A_1\subseteq
A_0$. Put $j_1=\ell_{i_1}^1$.

Put $m_2=\min\,\{m(\lambda)\mid\lambda\in A_1\}$. Let us fix a partition
$\lambda_2=(\ell_1^2,\ell_2^2,\dots,\ell_{m_2}^2)\in A_1$. The same arguments
as in the previous paragraph show that there is a number $i_2\le m_2$ and an
infinite set $A_2\subseteq A_1$  such that $n_{i_2}<\ell_{i_2}^2$ for every
$\nu=(n_1,n_2,\dots,n_k)\in A_2$. Put $j_2=\ell_{i_2}^2$.

Continuing this process, we construct a sequence of infinite sets of
partitions $A_0\supseteq A_1\supseteq A_2\supseteq\cdots$, a sequence of
partitions $\{\lambda_s=(\ell_1^s,\ell_2^s,\dots,\ell_{m_s}^s)\mid s\in
\mathbb N\}$, and two sequences of numbers $\{i_s\mid s\in\mathbb N\}$ and
$\{j_s\mid s\in\mathbb N\}$ such that, for any $s\in\mathbb N$, the following
holds: $\lambda_s\in A_{s-1}$, $i_s\le m_s$, $j_s=\ell_{i_s}^s$, and
$n_{i_s}<j_s$ for any $\nu=(n_1,n_2,\dots,n_k)\in A_s$. The choice of the
partitions $\lambda_1,\lambda_2,\dots$ guarantees that if $p>q$ then
$$\ell_{i_q}^p<\ell_{i_q}^q=j_q\ldotp$$
In particular, if $p>q$ and $i_p=i_q$ then
$$j_p=\ell_{i_p}^p=\ell_{i_q}^p<j_q\ldotp$$
This means that all pairs of the kind $(i_s,j_s)$ are different. Furthermore,
if $i_p\ge i_q$ then $\ell_{i_p}^p\le\ell_{i_q}^p$ because all partitions
$\lambda_1,\lambda_2,\dots$ are non-increasing sequences of numbers.
Therefore, if $p>q$ and $i_p\ge i_q$ then
$$j_p=\ell_{i_p}^p\le\ell_{i_q}^p<j_q\ldotp$$
Put $i_r=\min\,\{i_s\mid s\in\mathbb N\}$, $j_t=\min\,\{j_s\mid s\in\mathbb
N\}$, and $h=\max\,\{r,t\}$. If $s>h$ then $i_s\ge i_r$ and $j_s\ge j_t$,
whence $j_s<j_r$ and $i_s<i_t$. We see that both the sequences $\{i_s\mid s
\in\mathbb N\}$ and $\{j_s\mid s\in\mathbb N\}$ are bounded. But this is
impossible because all pairs of the kind $(i_s,j_s)$ are different. The
contradiction completes the proof.
\end{proof}

\section{Proof of Proposition \ref{reduce}: necessity}
\label{nec}

Here we aims to verify that if an overcommutative variety satisfies the
condition~(vi) of Theorem~\ref{main} then it is greedy. We start with some
new notation and several auxiliary facts.

For arbitrary words $w_1,w_2$ and an identity system $\Sigma$, we write $w_1
\stackrel{\Sigma}{\longrightarrow}w_2$ if there exist $a,b\in F^1$, $s,t\in
F$, and an endomorphism $\zeta$ on $F$ such that $w_1\equiv a\zeta(s)b$, $w_2
\equiv a\zeta(t)b$, and the identity $s\approx t$ belongs to $\Sigma$. It is
a common knowledge that an identity $u\approx v$ follows from a system
$\Sigma$ if and only if there exists a sequence of words $w_0,w_1,\dots,
w_\ell$ such that
\begin{equation}
\label{deduct}
u\equiv w_0\stackrel{\Sigma}{\longrightarrow}w_1\stackrel{\Sigma}
{\longrightarrow}\cdots\stackrel{\Sigma}{\longrightarrow}w_\ell\equiv v\ldotp
\end{equation}
This sequence is called a \emph{deduction} of the identity $u\approx v$ from
$\Sigma$. Note that if $\Sigma$ consists of balanced identities then $\ell
(w_i)=\ell(u\approx v)$, $n(w_i)=n(u\approx v)$, and $\partition(w_i)=
\partition(u\approx v)$ for all $i=0,1,\dots,\ell$.

\begin{lemma}
\label{poor xi}
Let $u$ be a word and $\xi$ an endomorphism on $F$ such that $\ell(\xi(u))=
\ell(u)$. Then $\partition(u)\preceq\partition(\xi(u))$.
\end{lemma}

\begin{proof}
Put $\lambda=\partition(u)$. It is clear that $\xi(x)$ is a letter for every
letter $x$. The requirement conclusion follows from the following evident
observation: $\partition(\xi(u))=U_S\bigl(\lambda^0\bigr)$ where $S$ is a
finite (may be empty) set of unions of components of $\lambda$ corresponding
to letters from $c(u)$ with the same image under $\xi$.
\end{proof}

\begin{lemma}
\label{lambda less u=v}
Let $\lambda=(\ell_1,\ell_2,\dots,\ell_m)\in\Lambda_{n,m}$. If a non-trivial
identity $u\approx v$ holds in the variety $S_\lambda$ then $\lambda\preceq
\partition(u\approx v)$.
\end{lemma}

\begin{proof}
We put
$$\Sigma=\{f\approx g\mid\ \text{there is}\ i\in\{0,1,\dots,s(\lambda)\}\
\text{with}\ f,g\in W_{n+i,m+i,\lambda^i}\}\ldotp$$
Thus, $S_\lambda=\var\Sigma$. Let~\eqref{deduct} be a deduction of the
identity $u\approx v$ from $\Sigma$. Note that $\ell\ge1$ because the
identity $u\approx v$ is non-trivial. We have $w_0\equiv a\zeta(s)b$ and $w_1
\equiv a\zeta(t)b$ for some homomorphism $\zeta$ on $F$, some $a,b\in F^1$,
and some $s,t\in W_{n+i,m+i,\lambda^i}$ where $i\in\{0,1,\dots,s(\lambda)\}$.

If $\ell(a\zeta(s)b)=\ell(s)$ then the words $a$ and $b$ are empty and
$\ell(\zeta(s))=\ell(s)$. Here we may apply Lemma~\ref{poor xi} and conclude
that
\begin{align*}
\lambda\preceq U_\varnothing\bigl(\lambda^i\bigr)&=\lambda^i=\partition(s)
\preceq\partition(\zeta(s))=\\
&=\partition(a\zeta(s)b)=\partition(w_0)=\partition(u\approx v)\ldotp
\end{align*}
Therefore, $\lambda\preceq\partition(u\approx v)$, and we are done.

Suppose now that $\ell(a\zeta(s)b)>\ell(s)$. For each $j=1,2,\dots,m+i$, we
denote by $y_j$ the first letter of the word $\zeta(x_j)$. Thus, $\zeta(x_j)
\equiv y_ju_j$ for some $u_j\in F^1$. We have
\begin{align*}
\partition(a\zeta(s)b)&=\partition(a(\zeta(x_1))^{\ell_1}\cdots(\zeta
(x_m))^{\ell_m}\zeta(x_{m+1})\cdots\zeta(x_{m+i})b)=\\
&=\partition(a(y_1u_1)^{\ell_1}\cdots(y_mu_m)^{\ell_m}y_{m+1}u_{m+1}\cdots
y_{m+i}u_{m+i}b)=\\
&=\partition(y_1^{\ell_1}\cdots y_m^{\ell_m}y_{m+1}\cdots y_{m+i}u_1^{\ell_1}
\cdots u_m^{\ell_m}u_{m+1}\cdots u_{m+i}ab)\ldotp
\end{align*}
Let $k=\ell(a\zeta(s)b)-\ell(s)$. Put
\begin{align*}
&c_1\equiv y_1^{\ell_1}\cdots y_m^{\ell_m}y_{m+1}\cdots y_{m+i};\\
&c_2\equiv u_1^{\ell_1}u_2^{\ell_2}\cdots u_m^{\ell_m}u_{m+1}\cdots u_{m+i}a
b;\\
&c\equiv c_1c_2\equiv y_1^{\ell_1}\cdots y_m^{\ell_m}y_{m+1}\cdots y_{m+i}
u_1^{\ell_1}\cdots u_m^{\ell_m}u_{m+1}\cdots u_{m+i}ab;\\
&d\equiv x_1^{\ell_1}\cdots x_m^{\ell_m}x_{m+1}\cdots x_{m+i+k}\ldotp
\end{align*}
Since $\partition(c)=\partition(a\zeta(s)b)$ and $\partition(c_1)=\partition
(s)$, we have $\ell(c)=\ell(a\zeta(s)b)$ and $\ell(c_1)=\ell(s)$. Besides
that, $\ell(c)=\ell(c_1)+\ell(c_2)$. Therefore,
$$\ell(c_2)=\ell(c)-\ell(c_1)=\ell(a\zeta(s)b)-\ell(s)=k\ldotp$$
It is convenient for us to rewrite the word $c_2$ in the form $c_2\equiv z_1
z_2\dots z_k$ where $z_1,z_2,\dots,z_k$ are (not necessarily different)
letters. Let $\xi$ be an endomorphism on $F$ such that
$$\xi(x_j)\equiv
\begin{cases}
y_j&\text{whenever}\ 1\le j\le m+i,\\
z_{j-m-i}&\text{whenever}\ m+i+1\le j\le m+i+k\ldotp
\end{cases}$$
Then $c\equiv\xi(d)$. It is clear that
$$\ell(d)=\ell(c_1)+k=\ell(s)+\ell(a\zeta(s)b)-\ell(s)=\ell(a\zeta(s)b)=\ell
(c)=\ell(\xi(d))\ldotp$$
Now we may apply Lemma~\ref{poor xi} and conclude that
\begin{align*}
\lambda\preceq U_\varnothing\bigl(\lambda^k\bigr)&=\lambda^k=\partition(d)
\preceq\partition(\xi(d))=\partition(c)=\\
&=\partition(a\zeta(s)b)=\partition(w_0)=\partition(u\approx v)
\ldotp
\end{align*}
Therefore, $\lambda\preceq\partition(u\approx v)$, and we are done.
\end{proof}

\begin{lemma}
\label{W_lambda in W_mu}
Let $\lambda,\mu\in\Lambda$. If $\mu=U_S(\lambda)$ for some finite set $S$
of unions of components then $\mathcal W_\lambda\subseteq\mathcal W_\mu$.
\end{lemma}

\begin{proof}
Let $\lambda=(\ell_1,\ell_2,\dots,\ell_m)\in\Lambda_{n,m}$. By the trivial
induction, it suffices to consider the case when $\mu=U_{i,j}(\lambda)$ for
some $i$ and $j$. We have to verify that if $u,v\in W_{n,m-1,\mu}$ then the
identity $u\approx v$ holds in $\mathcal W_\lambda$. Since $\partition(u
\approx v)=U_{i,j}(\lambda)$, there is a letter $x_k$ with $\ell_k(u)=\ell_k
(v)=\ell_i+\ell_j$. Let $x_p$ and $x_q$ be some letters such that $x_p,x_q
\notin c(u)$. One can change the first $\ell_i$ occurences of $x_k$ in $u$
and in $v$ by $x_p$, while the last $\ell_j$ occurences of $x_k$ in $u$ and
in $v$ by $x_q$. We obtain some identity $s\approx t$ with $\partition(s
\approx t)=\lambda$. Hence the variety $\mathcal W_{n,m,\lambda}$ satisfies
$s\approx t$. If we substitute $x_k$ for $x_p$ and $x_q$ in $s\approx t$ then
we return to the identity $u\approx v$. Therefore, $u\approx v$ follows from
$s\approx t$, whence $u\approx v$ holds in $\mathcal W_{n,m,\lambda}$.
\end{proof}

Recall that a letter $x_i$ is called \emph{simple in a word} $u$ if $\ell_i
(u)=1$.

\begin{lemma}
\label{s(lambda)=0}
If $\lambda\in\Lambda_{n,m}$ and $s(\lambda)=0$ then $\mathcal
W_{n,m,\lambda}\subseteq\mathcal W_{n+k,m+k,\lambda^k}$ for any positive
integer $k$.
\end{lemma}

\begin{proof}
The definition of the number $s(\lambda)$ immediately implies that if $s
(\lambda)=0$ then $s\bigl(\lambda^k\bigr)=0$ as well for any $k>0$. This
observation implies that, by the trivial induction, it suffices to consider
the case $k=1$.

First of all, we note that $\lambda\ne(2,1)$ because $s(\lambda)=1$
otherwise. For brevity, put $q=q(\lambda)$, $r=r(\lambda)$, $s=s(\lambda)$,
and $t=m-q$. Since $\lambda\ne(2,1)$, we have $\delta=1$ and therefore, $s=
\max\,\{r-q-1,0\}$. The equality $s=0$ implies now that $r-q-1\le0$, that is
\begin{equation}
\label{r<q+2}
r\le q+1\ldotp
\end{equation}
Suppose that $q\le1$. Then $r\le2$. Let $\lambda=(\ell_1.\ell_2,\dots,
\ell_m)$. The definition of the number $r(\lambda)$ and the inequality $r\le
2$ imply that either $\ell_1=\ell_2=\cdots=\ell_m=1$ or $\ell_1=2$ and
$\ell_2=\cdots=\ell_m=1$. Since $\lambda\ne(2,1)$, this implies that $q\ge2$.
We have a contradiction with the inequality $q\le1$. Therefore, $q\ge2$. This
implies that every word from the transversal $W_{n+1,m+1,\lambda^1}$ has at
least three simple letters.

We need to verify that any identity of the kind $u\approx v$ with $u,v\in
W_{n+1,m+1,\lambda^1}$ holds in $W_{n,m,\lambda}$. It suffices to check that
if $u\in W_{n+1,m+1,\lambda^1}$ then the variety $W_{n,m,\lambda}$ satisfies
the identity
\begin{equation}
\label{u=canon}
u\approx x_1^{\ell_1}\cdots x_t^{\ell_t}x_{t+1}\cdots x_{m+1}\ldotp
\end{equation}

At the rest part of the proof of this lemma, the words `a simple letter' mean
`a simple in $u$ letter'. One can note that one of the following three claims
hold:
\begin{itemize}
\item[1)]the word $u$ ends with a simple letter;
\item[2)]the word $u$ starts with a simple letter;
\item[3)]the word $u$ contains a subword of the kind $x_ix_j$ where $x_i$ and
$x_j$ are simple letters.
\end{itemize}
Indeed, if all these three claims fail then
$$u\equiv w_1y_1w_2y_2\cdots w_{q+1}y_{q+1}w_{q+2}$$
where $y_1,y_2,\dots,y_{q+1}$ are simple letters, while $w_1,w_2,\dots,
w_{q+2}$ are non-empty words such that the word $w\equiv w_1w_2\cdots
w_{q+2}$ does not contain simple letters. Then $r=\ell(w)=\sum
\limits_{i=1}^{q+2}\ell(w_i)\ge q+2$, contradicting the inequality~\eqref
{r<q+2}.

Now we consider three cases corresponding to the claims 1)--3).

\emph{Case} 1: $u\equiv wx_i$ for some word $w$ and some simple letter $x_i$.
The identity
\begin{equation}
\label{w=non-canon}
w\approx x_1^{\ell_1}\cdots x_t^{\ell_t}x_{t+1}\cdots x_{i-1}x_{i+1}\cdots
x_{m+1}
\end{equation}
has the partition $\lambda$, whence it holds in $\mathcal W_{n,m,\lambda}$.
Multiplying~\eqref{w=non-canon} by $x_i$ from the right, we have the identity
\begin{equation}
\label{u=last x_i}
u\approx x_1^{\ell_1}\cdots x_t^{\ell_t}x_{t+1}\cdots x_{i-1}x_{i+1}\cdots
x_{m+1}x_i
\end{equation}
that also holds in $\mathcal W_{n,m,\lambda}$. If $i=m+1$ then~\eqref
{u=last x_i} coincides with~\eqref{u=canon} and we are done. Let now $i\le
m$. Put
$$j=
\begin{cases}
m-1&\text{whenever}\ i=m,\\
m&\text{otherwise}\ldotp
\end{cases}$$
Since $u$ contains at least three simple letters, the letter $x_j$ is simple.
The identity~\eqref{u=last x_i} has the form
\begin{equation}
\label{u=aj,m+1,i}
u\approx ax_jx_{m+1}x_i
\end{equation}
for some $a\in F^1$. Let $x_p$ be a letter with $x_p\notin c(u)$. The
identity
\begin{equation}
\label{ap,i=ai,p}
ax_px_i\approx ax_ix_p
\end{equation}
has the partition $\lambda$, whence it holds in $\mathcal W_{n,m,\lambda}$.
Substituting $x_jx_{m+1}$ for $x_p$ in~\eqref{ap,i=ai,p}, we obtain the
identity
\begin{equation}
\label{aj,m+1,i=ai,j,m+1}
ax_jx_{m+1}x_i\approx ax_ix_jx_{m+1}
\end{equation}
that holds in $\mathcal W_{n,m,\lambda}$. The identity
\begin{equation}
\label{ai,j=short canon}
ax_ix_j\approx x_1^{\ell_1}\cdots x_t^{\ell_t}x_{t+1}\cdots x_m
\end{equation}
has the partition $\lambda$, whence it holds in $\mathcal W_{n,m,\lambda}$
too. Multiplying~\eqref{ai,j=short canon} on $x_{m+1}$ from the right, we
obtain the identity
\begin{equation}
\label{ai,j,m+1=canon}
ax_ix_jx_{m+1}\approx x_1^{\ell_1}\cdots x_t^{\ell_t}x_{t+1}\cdots x_{m+1}
\end{equation}
that holds in $\mathcal W_{n,m,\lambda}$ as well. Combining the
identities~\eqref{u=aj,m+1,i},~\eqref{aj,m+1,i=ai,j,m+1}, and~\eqref
{ai,j,m+1=canon}, we obtain the identity~\eqref{u=canon}.

\emph{Case} 2: $u\equiv x_iw$ for some simple letter $x_i$ and some word $w$.
The word $u$ contains at least three simple letters. Therefore, there is a
simple letter $x_j\in c(w)$. Thus, $w\equiv ax_jb$ for some $a,b\in F^1$. The
identity
\begin{equation}
\label{ajb=abj}
ax_jb\approx abx_j
\end{equation}
has the partition $\lambda$, whence it holds in $\mathcal W_{n,m,\lambda}$.
Multiplying~\eqref{ajb=abj} on $x_i$ from the left, we obtain the identity
$u\approx x_iabx_j$ that also holds in $\mathcal W_{n,m,\lambda}$. We come to
the situation considered in Case~1.

\emph{Case} 3: $u\equiv ax_ix_jb$ for some $a,b\in F^1$ and some simple
letters $x_i$ and $x_j$. Let $x_p$ be a letter with $x_p\notin c(u)$. The
identity
\begin{equation}
\label{apb=abp}
ax_pb\approx abx_p
\end{equation}
has the partition $\lambda$, whence it holds in $\mathcal W_{n,m,\lambda}$.
Substituting $x_ix_j$ for $x_p$ in~\eqref{apb=abp}, we obtain the identity
$u\approx abx_ix_j$ that holds in $\mathcal W_{n,m,\lambda}$. We come to the
situation considered in Case~1 again.
\end{proof}

\begin{corollary}
\label{S_lambda in W_lambda^k}
If $\lambda\in\Lambda$ then $\mathcal{S_\lambda\subseteq W}_{\lambda^k}$ for
any $k\ge0$.
\end{corollary}

\begin{proof}
If $k\le s(\lambda)$ then the desired inclusion holds by the definition of
the variety $\mathcal S_\lambda$. Let now $k>s(\lambda)$. It is easy to see
that $s\bigl(\lambda^{s(\lambda)}\bigr)=0$. Now we may apply Lemma~\ref
{s(lambda)=0} and conclude that $\mathcal{S_\lambda\subseteq
W}_{\lambda^{s(\lambda)}}\subseteq\mathcal W_{\lambda^k}$.
\end{proof}

\begin{proposition}
\label{S_lambda is greedy}
If $\lambda\in\Lambda$ then the variety $\mathcal S_\lambda$ is greedy.
\end{proposition}

\begin{proof}
Suppose that $\mu\in\Lambda$ and the variety $\mathcal S_\lambda$ reduces the
transversal $W_\mu$, that is $\mathcal S_\lambda$ satisfies some non-trivial
identity $u\approx v$ with $u,v\in W_\mu$. Lemma~\ref{lambda less u=v}
implies that
$\lambda\preceq\mu$, that is $\mu=U_S\bigl(\lambda^k\bigr)$ for some $S$ and
$k$. Applying Corollary~\ref{S_lambda in W_lambda^k} and Lemma~\ref
{W_lambda in W_mu}, we have $\mathcal{S_\lambda\subseteq W}_{\lambda^k}
\subseteq\mathcal W_\mu$. Thus, if $s,t\in W_\mu$ then the identity $s\approx
t$ holds in $\mathcal S_\lambda$. This means that $\mathcal S_\lambda$
collapses $W_\mu$. We see that the variety $\mathcal S_\lambda$ collapses any
transversal it reduces, that is $\mathcal S_\lambda$ is greedy.
\end{proof}

Now we are well prepared to prove the `only if' part of Proposition \ref
{reduce}. Let an overcommutative variety $\mathcal V$ satisfy the
condition~(vi) of Theorem~\ref{main}. We need to verify that $\mathcal V$ is
greedy. It is evident that the variety $\mathcal{SEM}$ is greedy because it
does not reduce any transversal. Let now $\mathcal V=\bigwedge\limits_{i=1}^k
\mathcal S_{\lambda_i}$ for some partitions $\lambda_1,\lambda_2,\dots,
\lambda_k$. By Proposition~\ref{S_lambda is greedy}, the varieties $\mathcal
S_{\lambda_1}$, $\mathcal S_{\lambda_2}$, \dots, $\mathcal S_{\lambda_k}$ are
greedy. Proposition~\ref{greedy} implies now that all these varieties are
neutral elements of the lattice \textbf{OC}. It is well known that the set of
all neutral elements of a lattice $L$ forms a sublattice of $L$ (see~\cite
[Theorem~III.2.9]{Gratzer-98}, for instance). Therefore, $\mathcal V$ is a
neutral element of \textbf{OC}. Now we may apply Proposition~\ref{greedy}
again and conclude that $\mathcal V$ is greedy.

\section{Proof of Proposition \ref{reduce}: sufficiency}
\label{suff}

Here we are going to verify that a greedy overcommutative variety satisfies
the condition (vi) of Theorem~\ref{main}, thus completing the proof of
Proposition~\ref{reduce} and therefore, of Theorem~\ref{main}. We start with
a few easy observations.

\begin{lemma}
\label{if collapses then reduces}
If an overcommutative semigroup variety $\mathcal V$ reduces \textup(in
particular, collapses\textup) a transversal $W_\lambda$ then $\mathcal V$
reduces transversals $W_{\lambda^k}$ for all $k\ge0$.
\end{lemma}

\begin{proof}
The case $k=0$ is obvious because $W_{\lambda^0}=W_\lambda$. Let now $k>0$.
Suppose that $\mathcal V$ satisfies a non-trivial identity of the kind $u
\approx v$ with $u,v\in W_\lambda$. Let $y_1,\dots,y_k$ be letters with $y_1,
\dots,y_k\notin c(u)$. The identity $uy_1\cdots y_k\approx vy_1\cdots y_k$ is
non-trivial and holds in $\mathcal V$ because it follows from $u\approx v$.
Since
$$\partition(uy_1\cdots y_k\approx vy_1\cdots y_k)=\lambda^k,$$
$\mathcal V$ reduces $W_{\lambda^k}$.
\end{proof}

\begin{lemma}
\label{V in S_lambda}
Let $\mathcal V$ be a greedy variety. If a non-trivial identity $u\approx v$
holds in $\mathcal V$ and $\partition(u\approx v)=\lambda$ then $\mathcal{V
\subseteq S_\lambda}$.
\end{lemma}

\begin{proof}
By Lemma~\ref{if collapses then reduces}, $\mathcal V$ reduces transversals
$W_{\lambda^k}$ for all $k=0,1,\dots,s(\lambda)$. Being greedy, $\mathcal V$
collapses all these transversals. Therefore, $\mathcal{V\subseteq
S_\lambda}$.
\end{proof}

\begin{corollary}
\label{S_lambda in S_mu}
Let $\lambda,\mu\in\Lambda$. Then $\mathcal{S_\lambda\subseteq S_\mu}$ if and
only if $\lambda\preceq\mu$.
\end{corollary}

\begin{proof}
\emph{Necessity}. Suppose that $\mathcal{S_\lambda\subseteq S_\mu}$. Let $u
\approx v$ be an identity with $\partition(u\approx v)=\mu$. Then $u\approx
v$ holds in $\mathcal S_\mu$, whence it holds in $\mathcal S_\lambda$. Now
Lemma~\ref{lambda less u=v} applies with the conclusion that $\lambda\preceq
\mu$.

\emph{Sufficiency}. Let $\lambda\preceq\mu$ and $u\approx v$ an identity with
$\partition(u\approx v)=\lambda$. Since $\lambda\preceq\mu$, there is an
identity $s\approx t$ such that $u\approx v$ implies $s\approx t$ and
$\partition(s\approx t)=\mu$. The variety $\mathcal S_\lambda$ satisfies the
identity $u\approx v$. Hence $s\approx t$ holds in $\mathcal S_\lambda$ as
well. According to Proposition~\ref{S_lambda is greedy}, the variety
$\mathcal S_\lambda$ is greedy. Now Lemma~\ref{V in S_lambda} succsessfully
applies with the conclusion that $\mathcal{S_\lambda\subseteq S_\mu}$.
\end{proof}

Now we are ready to prove the `if' part of Proposition \ref{reduce}. Let
$\mathcal V$ be a greedy variety and $\mathcal{V\ne SEM}$. The last
unequality means that $\mathcal V$ satisfies some non-trivial identity $u
\approx v$. Put $\lambda=\partition(u\approx v)$. Lemma~\ref{V in S_lambda}
shows that $\mathcal{V\subseteq S_\lambda}$. Thus, the set $\Gamma=\{\lambda
\in\Lambda\mid\mathcal{V\subseteq S_\lambda}\}$ is non-empty. Put $\mathcal
X=\bigwedge\limits_{\lambda\in\Gamma}\mathcal S_\lambda$. Clearly, $\mathcal
{V\subseteq X}$. Suppose that $\mathcal{V\ne X}$. Then there is an identity
$u\approx v$ that holds in $\mathcal V$ but fails in $\mathcal X$. Let $\mu=
\partition(u\approx v)$. By Lemma~\ref{V in S_lambda}, $\mathcal{V\subseteq
S_\mu}$. This means that $\mu\in\Gamma$, whence $\mathcal{X\subseteq S_\mu}$.
Since $\mathcal W_\mu$ satisfies $u\approx v$ and $\mathcal{X\subseteq S_\mu
\subseteq W_\mu}$, we have that the identity $u\approx v$ holds in $\mathcal
X$. A contradiction shows that $\mathcal{V=X}$.

Lemma~\ref{min condition} and Corollary~\ref{S_lambda in S_mu} imply together
that $\mathcal V=\bigwedge\limits_{\lambda\in\Gamma'}\mathcal S_\lambda$
where $\Gamma'$ is the set of all minimal elements of the partially ordered
set $\langle\Gamma;\preceq\rangle$. Since $\Gamma'$ forms an anti-chain in
$\langle\Lambda;\preceq\rangle$, Lemma~\ref{antichains} implies that the set
$\Gamma'$ is finite. Thus, $\mathcal V$ satisfies the condition~(vi) of
Theorem~\ref{main}.

Proposition~\ref{reduce} and Theorem~\ref{main} are proved.\qed

\section{Additional remarks}
\label{optim}

Here we are going to show that the description of the varieties under
consideration given by Theorem~\ref{main} may not be improved, in a sense.
Theorem~\ref{main} shows that the varieties of the kind $\mathcal S_\lambda$
play the crucial role in the description of varieties we consider in this
article. Recall that the variety $\mathcal S_\lambda$ is defined as the
intersection of the varieties $\mathcal W_{\lambda^i}$ where $i$ runs over
the set $\{0,1,\dots,s(\lambda)\}$. A natural question arises, whether or not
the number $s(\lambda)$ may be changed on some lesser number here.

For any $\lambda\in\Lambda$ and $k\in\{0,1,\dots,s(\lambda)\}$, we put
$$\mathcal S_\lambda^k=\bigwedge_{i=0}^k\mathcal W_{\lambda^i}\ldotp$$
In particular, $\mathcal S_\lambda^0=\mathcal W_\lambda$ and $\mathcal
S_\lambda^{s(\lambda)}=\mathcal S_\lambda$. The crucial property of the
variety $\mathcal S_\lambda$ is given by Proposition~\ref
{S_lambda is greedy}: this variety is greedy. The following statement
together with Lemma~\ref{if collapses then reduces} show that varieties
$\mathcal S_\lambda^k$ with $k<s(\lambda)$ does not have this property. Thus,
the question posed in the previous paragraph is answered in negative.

\begin{proposition}
\label{optimum}
Let $\lambda\in\Lambda$, $s(\lambda)>0$, and $0\le k<s(\lambda)$. Then the
variety $\mathcal S_\lambda^k$ does not collapse the transversals
$W_{\lambda^{k+1}}$, $W_{\lambda^{k+2}}$, \dots, $W_{\lambda^{s(\lambda)}}$.
\end{proposition}

\begin{proof}
Let $i\in\{k+1,\dots,s(\lambda)\}$. Suppose that $\mathcal S_\lambda^k$
collapses the transversal $W_{\lambda^i}$. Further considerations are divided
into two cases.

\emph{Case} 1: $\lambda\ne(2,1)$. The definition of the number $s(\lambda)$
and the inequality $s(\lambda)>0$ imply that $s(\lambda)=r(\lambda)-q
(\lambda)-1$ here. Since $s(\lambda)\ge i$, we have $r(\lambda)-q(\lambda)-1
\ge i$. Evident equalities $r\bigl(\lambda^i\bigr)=r(\lambda)$ and $q\bigl
(\lambda^i\bigr)=q(\lambda)+i$ then imply that $r\bigl(\lambda^i\bigr)\ge q
\bigl(\lambda^i\bigr)+1$. Hence the transversal $W_{\lambda^i}$ contains a
word $u$ of the kind
$$u\equiv w_1y_1w_2y_2\cdots w_{q+i}y_{q+i}w_{q+i+1}$$
where $y_1,y_2,\dots,y_{q+i}$ are simple in $u$ letters, while $w_1,w_2,
\dots,w_{q+i+1}$ are non-empty words such that the word $w_1w_2\cdots
w_{q+i+1}$ does not contain simple in $u$ letters. Let $v\in W_{\lambda^i}$
and $v\not\equiv u$. Since $\mathcal S_\lambda^k$ collapses $W_{\lambda^i}$,
the identity $u\approx v$ holds in $S_\lambda^k$. Therefore, this identity
follows from the identity system
$$\Sigma=\{g\approx h\mid\ \text{there is}\ j\in\{0,1,\dots,k\}\ \text{with}\
g,h\in W_{\lambda^j}\}\ldotp$$
Let~\eqref{deduct} be a deduction of $u\approx v$ from $\Sigma$. We have $u
\equiv a\zeta(s)b$ and $w_1\equiv a\zeta(t)b$ for some homomorphism $\zeta$
on $F$, some $a,b\in F^1$, and some $s,t\in W_{\lambda^j}$ where $j\in\{0,1,
\dots,k\}$. Furthermore,
\begin{equation}
\label{equalities}
r(\partition(s))=r\bigl(\lambda^j\bigr)=r\bigl(\lambda^i\bigr)=r(\partition
(u))=r(\partition(a\zeta(s)b))\ldotp
\end{equation}
On the other hand, it is evident that
\begin{equation}
\label{unequalities}
r(\partition(s))\le r(\partition(\zeta(s)))\le r(\partition(a\zeta(s)b))
\ldotp
\end{equation}
Combining~\eqref{equalities} and~\eqref{unequalities}, we have
\begin{equation}
\label{one more equalities}
r(\partition(s))=r(\partition(\zeta(s)))=r(\partition(a\zeta(s)b))\ldotp
\end{equation}
This implies that the subword $\zeta(s)$ of the word $a\zeta(s)b$ contains
all occurrences of non-simple in $a\zeta(s)b$ letters, whence all letters
from $c(ab)$ are simple in $u$. But the word $a\zeta(s)b\equiv u$ starts and
ends with non-simple in $u$ letters. Therefore, the words $a$ and $b$ are
empty. Thus, $u\equiv\zeta(s)$. If either there is a non-simple in $s$ letter
$x$ with $\ell(\zeta(x))>1$ or there is a simple in $s$ letter $y$ such that
the word $\zeta(y)$ contains some non-simple in $u$ letter then $r(\partition
(\zeta(s)))>r(\partition(s))$, contradicting~\eqref{one more equalities}.
Hence $\zeta$ maps every non-simple in $s$ letter to a letter and maps every
simple in $s$ letter to a word consisting of simple letters. If $y$ is a
simple in $s$ letter then $\zeta(y)$ is a subword of $u$. But $u$ does not
contain subwords consisting of simple letters except subwords of length~1,
that is letters. Thus, $\zeta$ maps a simple in $s$ letter to a simple in
$\zeta(s)\equiv u$ letter. We conclude that $\ell(u)=\ell(\zeta(s))=\ell(s)$.
But $\ell(s)=n+j$ and $\ell(u)=n+i$ where $n=\ell(\lambda)$. Therefore, $j=
i$. But this is impossible because $i\ge k+1$, while $j\le k$.

\emph{Case} 2: $\lambda=(2,1)$. Here $r(\lambda)=2$, $q(\lambda)=1$, and
$\delta=0$, whence $s(\lambda)=1$. Therefore, $k=0$ and $i=1$. This means
that $\mathcal S_\lambda^k=\mathcal W_\lambda$ and $\lambda^i=(2,1,1)$. Let
$u\equiv x^2yz$ and $v\equiv x^2zy$. Suppose that the identity $u\approx v$
holds in $\mathcal W_\lambda$. Then it follows from the identity system
$$\Sigma=\{x^2y\approx xyx\approx yx^2\}\ldotp$$
Let~\eqref{deduct} be a deduction of $u\approx v$ from $\Sigma$. Then there
is $j\in\{0,1,\dots,\ell\}$ such that the first occurrence of $z$ in the word
$w_j$ precedes the first occurrence of $y$ in $w_j$. Let $j$ be the least
number with such a property. It is evident that $j>0$. Thus, the following
holds:
\begin{align}
\label{w_j-1}&w_{j-1}\in\{x^2yz,xyxz,xyzx,yx^2z,yxzx,yzx^2\},\\
\label{w_j}&w_j\in\{x^2zy,xzxy,xzyx,zx^2y,zxyx,zyx^2\}\ldotp
\end{align}
Furthermore, $w_{j-1}\equiv a\zeta(s)b$ and $w_j\equiv a\zeta(t)b$ for some
homomorphism $\zeta$ on $F$, some $a,b\in F^1$, and some $s,t\in\{x^2y,xyx,y
x^2\}$. Repeating mutatis mutandi arguments from Case~1), we obtain that $r
(\partition(w_{j-1}))=r(\partition(w_j))=2$ and deduct from these equalities
that $x\notin c(ab)$, $\zeta(x)$ is a letter, and $\zeta(y)\in\{y,z,yz,zy\}$.
If $\zeta(x)\equiv e$ then $e$ is a non-simple in $w_{j-1}$ letter. In view
of~\eqref{w_j-1}, $\zeta(x)\equiv x$.

If $\zeta(y)\equiv yz$ then the word $w_j$ contains the subword $yz$,
contradicting~\eqref{w_j}. Analogously, if $\zeta(y)\equiv zy$ then the word
$w_{j-1}$ contains the subword $zy$, contradicting~\eqref{w_j-1}.

Suppose now that $\zeta(y)\equiv y$. Then $\zeta(s)\equiv s$ and $\zeta(t)
\equiv t$. Since $\zeta(s)$ is a subword in $w_{j-1}$, this means that one of
the words $x^2y$, $xyx$ or $yx^2$ is a subword in $w_{j-1}$. In view
of~\eqref{w_j-1}, this means that $w_{j-1}$ coincides with one of the words
$x^2yz$, $xyxz$ or $yx^2z$. Thus, the word $a$ is empty and therefore, $w_j
\equiv\zeta(t)b\equiv tb$. Since $z\notin c(t)$, we have that the first
occurrence of $y$ in $w_j$ precedes the first occurrence of $z$ in $w_j$. But
this contradicts the choice of the number $j$.

Finally, let $\zeta(y)\equiv z$.  Then $\zeta(s)\in\{x^2z,xzx,zx^2\}$, whence
$w_{j-1}$ coincides with one of the words $ax^2zb$, $axzxb$ or $azx^2b$. In
view of~\eqref{w_j-1}, this means that $w_{j-1}\in\{yx^2z,yxzx,yzx^2\}$.
Therefore, $a\equiv y$. Thus, the word $w_j$ starts with the letter $y$. As
in the previous paragraph, we have that the first occurrence of $y$ in $w_j$
precedes the first occurrence of $z$ in $w_j$ that contradicts the choice of
$j$.

We prove that the variety $\mathcal{W_\lambda=S}_\lambda^k$ does not satisfy
the identity $x^2yz\approx x^2zy$. Since $\partition(x^2yz\approx x^2zy)=(2,
1,1)=\lambda^i$, we have that $\mathcal S_\lambda^k$ does not collapse the
transversal $W_{\lambda^i}$.
\end{proof}

One can return to the definition of the variety $\mathcal S_\lambda$. It may
be written in the form
\begin{equation}
\label{definition of S__lambda}
\mathcal S_\lambda=\bigwedge_{\mu\in\Gamma}\mathcal W_\mu
\end{equation}
where $\Gamma=\{\lambda^k\mid k=0,1,\dots,s(\lambda)\}$. The following
assertion shows that the set $\{\lambda^k\mid k=0,1,\dots,s(\lambda)\}$ is
the least set of partitions $\Gamma$ such that the equality~\eqref
{definition of S__lambda} holds.

\begin{corollary}
\label{representation of S_lambda}
If the equality~\textup{\eqref{definition of S__lambda}} holds for some
$\Gamma\subseteq\Lambda$ then $\lambda^k\in\Gamma$ for all $k=0,1,\dots,s
(\lambda)$.
\end{corollary}

\begin{proof}
Suppose that $\lambda^k\notin\Gamma$ for some $k\in\{0,1,\dots,s(\lambda)\}$.
Let $u,v\in W_{\lambda^k}$. The definition of the variety
$\mathcal S_\lambda$ implies that the identity $u\approx v$ holds in
$\mathcal S_\lambda$. Therefore, this identity follows from the identity
system
$$\Sigma=\{g\approx h\mid\text{there is}\ \mu\in\Gamma\ \text{such that}\ g,h
\in W_\mu\}\ldotp$$
As usual, let~\eqref{deduct} be a deduction of $u\approx v$ from $\Sigma$.
Let $i\in\{0,1,\dots,\ell-1\}$. Then the identity $w_i\approx w_{i+1}$
follows from an identity of the kind $s\approx t$ where $s,t\in W_\mu$ for
some $\mu\in\Gamma$. The identity $s\approx t$ holds in the variety $\mathcal
S_\mu$. Therefore, $u\approx v$ holds in $\mathcal S_\mu$ too. Then
Lemma~\ref{lambda less u=v} implies that $\mu\preceq\partition(u\approx v)=
\lambda^k$. Furthemore, the identity $s\approx t$ holds in the variety
$\mathcal S_\lambda$ because $\mathcal{S_\lambda\subseteq W_\mu}$. Applying
Lemma~\ref{lambda less u=v} again, we have $\lambda\preceq\mu$. Therefore,
$\mu=U_S\bigl(\lambda^j\bigr)$ and $\lambda^k=U_T\bigl(\mu^q\bigr)$ for some
finite (may be empty) sets of partitions $S$ and $T$ and some non-negative
integers $j$ and $q$.

Let $n(\lambda)=n$. Then $n(\mu)=n+j$, while $n\bigl(\lambda^k\bigr)$ equals
both $n+k$ (that is evident) and $n+j+q$ (because $n\bigl(\lambda^k\bigr)=n
(\mu)+q$). Therefore, $n+k=n+j+q$, whence $q=k-j$. Thus, $\mu=U_S\bigl
(\lambda^j\bigr)$ and $\lambda^k=U_T\bigl(\mu^{k-j}\bigr)$. Hence $\lambda^k=
U_{S\cup T}\bigl(\lambda^k\bigr)$ and therefore, $S=T=\varnothing$. In
particular, this means that $\mu=U_\varnothing\bigl(\lambda^j\bigr)=
\lambda^j$. We see that $\lambda^j=\mu\preceq\lambda^k$, whence $j\le k$. But
$j\ne k$ because $\lambda^k\notin\Gamma$, while $\lambda^j=\mu\in\Gamma$.
Hence $j<k$. Besides that, the equality $\mu=\lambda^j$ implies that the
identity $w_i\approx w_{i+1}$ holds in the variety $\mathcal W_{\lambda^j}$.
Since $j\le k-1$, we have $\mathcal S_\lambda^{k-1}\subseteq\mathcal
W_{\lambda^j}$. Thus, $w_i\approx w_{i+1}$ holds in $\mathcal
S_\lambda^{k-1}$. This is the case for all $i=0,1,\dots,\ell-1$. Therefore,
the identity $u\approx v$ holds in $\mathcal S_\lambda^{k-1}$ too. This is
valid for all $u,v\in W_{\lambda^k}$. Hence the variety $\mathcal
S_\lambda^{k-1}$ collapses the transversal $W_{\lambda^k}$, contradicting
Proposition~\ref{optimum}.
\end{proof}

\section{Open problems}
\label{quest}

Recall that an element $x$ of a lattice $L$ is called \emph{modular} if
$$\forall y,z\in L\colon\quad y\le z\longrightarrow(x\vee y)\wedge z=(x\wedge
z)\vee y,$$
and \emph{upper-modular} if
$$\forall y,z\in L\colon\quad y\le x\longrightarrow x\wedge(y\vee z)=(x\wedge
z)\vee y\ldotp$$
\emph{Lower-modular} elements are defined dually to upper-modular ones.

\begin{problem}
\label{mod in OC}
Describe
\begin{itemize}
\item[\textup{a)}]modular;
\item[\textup{b)}]upper-modular;
\item[\textup{c)}]lower-modular
\end{itemize}
elements of the lattice $\mathbf{OC}$.
\end{problem}

As we have already mentioned in Section~\ref{intr}, neutral elements of the
lattice \textbf{SEM} are completely determined in~\cite{Volkov-05}, while
distributive elements of this lattice are completely described in~\cite
{Vernikov-Shaprynskii-distr}. It is interesting to note that Proposition~\ref
{greedy} plays an important role in the proof of the result of~\cite
{Vernikov-Shaprynskii-distr}.

\begin{problem}
\label{distr in SEM}
Describe
\begin{itemize}
\item[\textup{a)}]codistributive;
\item[\textup{b)}]standard;
\item[\textup{c)}]costandard
\end{itemize}
elements of the lattice $\mathbf{SEM}$.
\end{problem}

Some particular results concerning Problem~\ref{distr in SEM}a) are obtained
in~\cite{Vernikov-codistr}. The following two examples show that, in contrast
with the overcommutative case, the lattice \textbf{SEM} contains distributive
but not codistributive elements and codistributive but not distributive ones.
In particular, there are [co]distributive but not neutral elements of \textbf
{SEM}.

\begin{example}
\label{distr not codistr}
The variety $\mathcal N=\var\{x^2y=xyx=yx^2=0\}$ is a distributive element of
the lattice \textbf{SEM} by~\cite[Theorem~1.1]{Vernikov-Shaprynskii-distr}.
But this variety is not a codistributive (and moreover not a neutral) element
of \textbf{SEM} by~\cite[Theorem~1.1]{Vernikov-codistr}.
\end{example}

\begin{example}
\label{codistr not distr}
The varieties $\mathcal A_p=\var\{x^py=y,xy=yx\}$ with any prime $p$,
$\mathcal{LZ}=\var\{xy=x\}$, and $\mathcal{RZ}=\var\{xy=y\}$ are
codistributive elements of the lattice \textbf{SEM}. This follows from the
well known facts that these varieties are atoms of \textbf{SEM} and \textbf
{SEM} satisfies the condition
$$\forall x,y,z\colon\quad x\wedge z=y\wedge z=0\longrightarrow(x\vee y)
\wedge z=0$$
(see~\cite[Section~1]{Shevrin-Vernikov-Volkov-09}, for instance). But
$\mathcal A_p$, $\mathcal{LZ}$, and $\mathcal{RZ}$ are not distributive (and
moreover not neutral) elements of \textbf{SEM} by~\cite[Theorem~1.1]
{Vernikov-Shaprynskii-distr}.
\end{example}

\end{document}